\documentclass[11pt]{article}

\usepackage[a4paper,margin=1in]{geometry}
\usepackage{amsmath,amssymb,amsthm,mathtools}
\usepackage{microtype}
\usepackage{cite}
\usepackage[colorlinks=true,linkcolor=blue,citecolor=blue,urlcolor=blue]{hyperref}

\allowdisplaybreaks

\newtheorem{theorem}{Theorem}[section]
\newtheorem{lemma}[theorem]{Lemma}
\newtheorem{proposition}[theorem]{Proposition}

\newtheorem{definition}[theorem]{Definition}
\newtheorem{remark}[theorem]{Remark}

\newcommand{\mbX}{\mathcal X}
\newcommand{\mbY}{\mathcal Y}

\newcommand{\wt}{\operatorname{wt}}

\newcommand{\E}{\mathbb E}

\newcommand{\one}{\boldsymbol{1}}
\newcommand{\bx}{\boldsymbol{x}}
\newcommand{\by}{\boldsymbol{y}}

\newcommand{\bX}{\boldsymbol{X}}
\newcommand{\bY}{\boldsymbol{Y}}

\title{Improved Upper Bound for Lindstr\"om's Unique-Sum Problem}

\author{
Ruze Zhang
\thanks{Institute of Network Coding and Department of Information Engineering,
The Chinese University of Hong Kong,
Hong Kong SAR, China.
Email: \texttt{rzzhang@mail.nankai.edu.cn}.}
}

\date{}

\begin{document}
\maketitle

\begin{abstract}
Let $\{0,1\}^{n}$ denote the set of all $n$-dimensional vectors whose components are either $0$ or~$1$.
For any two nonempty subsets $\mathcal{X},\mathcal{Y}\subseteq\{0,1\}^{n}$,
the pair
$(\mathcal X,\mathcal Y)$  is called an $n$-dimensional
\emph{unique-sum pair} if each pair
$(\boldsymbol x,\boldsymbol y)\in\mathcal X\times\mathcal Y$
can be uniquely determined from the arithmetic sum
$\boldsymbol x+\boldsymbol y$.
Let $M(n)$ denote the maximum value of
$|\mathcal X||\mathcal Y|$ over all $n$-dimensional unique-sum pairs
$(\mathcal X,\mathcal Y)$.
In 1969, Lindstr\"om proved that
$$
\frac{1}{2}(1+\log_2 3)
\leq \lim_{n\to\infty}\frac{1}{n}\log_2 M(n)
\leq 3/2.
$$
Since then, the lower bound $(1+\log_2 3)/2$ has been successively improved, whereas the
upper bound $3/2$ has remained unchanged.
In this paper, we establish an explicit upper bound whose numerical value
is approximately $1.4884$.
To the best of our knowledge, this is the first strict improvement over
Lindstr\"om's upper bound $3/2$.
Towards this end, we develop a \emph{coordinate projection approach} that constructs a lower-dimensional \emph{unique-sum system} from a unique-sum pair.
By combining the results obtained from this approach with a relaxed
form of a necessary condition established by Ordentlich and Shayevitz
on unique-sum systems, we establish the above improved upper bound.
\end{abstract}

\section{Problem Formulation and Main Result}

Let $\{0,1\}^{n}$ denote the set of all $n$-dimensional vectors whose components are either $0$ or $1$.
For any two nonempty subsets $\mathcal{X},\mathcal{Y}\subseteq\{0,1\}^{n}$,  let
\begin{align*}
\mathcal{X}+\mathcal{Y}\triangleq\{\boldsymbol{x}+\boldsymbol{y}:~\boldsymbol{x}\in\mathcal{X},~\boldsymbol{y}\in\mathcal{Y}\}\subseteq \{0,1,2\}^{n},
\end{align*}
where $\boldsymbol{x}+\boldsymbol{y}$ denotes the componentwise \emph{arithmetic sum} of
$\boldsymbol{x}$ and $\boldsymbol{y}$.
Further,
the pair $(\mathcal{X},\mathcal{Y})$    is
called an \emph{$n$-dimensional unique-sum pair} if
the mapping $
(\boldsymbol{x},\boldsymbol{y})\longmapsto
\boldsymbol{x}+\boldsymbol{y}$
is injective on
$\mathcal{X}\times\mathcal{Y}$, or equivalently,
$
|\mathcal{X}||\mathcal{Y}|=
|\mathcal{X}+\mathcal{Y}|.$
For every positive integer~$n$, define
\begin{align}
M(n)
\triangleq
\max\Big\{
 |\mbX||\mbY|: ~\textup{$(\mathcal{X},\mathcal{Y})$    is
  an  $n$-dimensional unique-sum pair}
\Big\}.
\label{eq:def-Mn}
\end{align}

In 1969, Lindstr\"om~\cite{Lindstrom1969} considered the problem of estimating $M(n)^{1/n}$ as $n\to\infty$.
For this problem, Lindstr\"om~\cite[Thm.~1]{Lindstrom1969} showed that the limit
$
\lim_{n\to\infty}M(n)^{1/n}$
exists and established
the following classical lower and upper bounds on  this limit:
\begin{align}
6^{1/2}
\leq
\lim_{n\to\infty}M(n)^{1/n}
\leq 8^{1/2},\label{origi-bound}
\end{align}
where the upper bound $8^{1/2}$ was obtained by using an
information-theoretic  approach.
For convenience, the base-$2$ logarithm of   $\lim_{n\to\infty}M(n)^{1/n}$ is
referred to as the \emph{optimal rate} for unique-sum pairs and is
denoted by
\begin{align}
R
\triangleq\log \Big(\lim_{n\to\infty}M(n)^{1/n}\Big)=
\lim_{n\to\infty}\frac{1}{n}\log M(n).
\label{eq:def-Csum}
\end{align}
Throughout this paper, ``$\log$'' denotes ``$\log_2$'', the logarithm with base~$2$.
Accordingly, Lindstr\"om's bounds in \eqref{origi-bound} can be written as
\begin{align*}
\frac{1}{2}(1+\log 3)
\leq
R
\leq
\frac{3}{2}.
\end{align*}

The problem of determining the optimal rate $R$ for unique-sum pairs has been studied from
both combinatorial and information-theoretic perspectives.
In information theory,
the optimal rate $R$ arises naturally in the study of zero-error coding for the two-user binary adder channel, a fundamental multiple-access channel that
has been extensively studied since the late 1970s~\cite{Weldon1978}.
More precisely, a unique-sum pair is usually referred to as a
\emph{uniquely decodable code pair}
for the two-user binary adder channel, and the optimal rate $R$
is precisely the \emph{maximum zero-error sum rate} of this channel.
Determining the exact maximum zero-error sum rate remains a longstanding open problem.
Although several outer bounds on the \emph{zero-error capacity region} have been obtained by using various combinatorial and information-theoretic methods
\cite{Weldon1978,vanTilborg1983,UrbankeLi1998,
MattasOstergard2005,OrdentlichShayevitz2016,
AustrinEtAl2018,YuAnantharamChen2024}, none of them improves Lindstr\"om's upper bound $3/2$
on the maximum zero-error sum rate.
Consequently,
whether Lindstr\"om's upper bound $3/2$ could be strictly improved has remained an open problem for more than fifty years.

On the other hand, a series of constructions of unique-sum pairs  have been proposed to
improve the lower bound on the optimal rate~$R$, e.g.,
\cite{Weldon1978,Khachatrian1982,KasamiEtAl1983,
CoeberghVanDenBraakVanTilborg1985,AhlswedeBalakirsky1999,
MattasOstergard2005,BaloghEtAl2025}.
Recently, Balogh et al.~\cite{BaloghEtAl2025} obtained the best currently
known lower bound
$
R\geq 1.3184.
$
Thus, before the
present work, the best known bounds on the optimal rate~$R$ were
\begin{align*}
1.3184
\leq
R
\leq
\frac{3}{2}.
\end{align*}

In this paper, we establish a strict improvement over Lindstr\"om's  upper bound.
Before presenting this result, we introduce some notation.
The \emph{binary entropy function} is denoted by
\begin{align*}
h(p)
\triangleq
-p\log p-(1-p)\log(1-p),
\qquad
p\in[0,\,1],
\end{align*}
where $0\log0$ is taken to be~$0$.
Further, define
\begin{align}
\Gamma(\alpha)
\triangleq
1+h(\alpha)
-
2h\left(
\frac{1-\sqrt{1-2\alpha}}{2}
\right),
\qquad
\alpha\in[0,\,1/2].
\label{eq:def-h-hat}
\end{align}
It will be shown in Lemma~\ref{lem:function-properties} of
Section~\ref{sec:system-bound} that
$\Gamma$ is strictly decreasing on $[1/3,\,1/2]$
with
$\Gamma(1/3)>1/3$ and $\Gamma(1/2)=0$.
Hence, we let $\Gamma^{-1}$ denote
the inverse function of the restriction of~$\Gamma$ to $[1/3,\,1/2]$, which is well defined. In particular, we have
\begin{align}
\Gamma^{-1}(\beta)\in(1/3,\,1/2],
\qquad
\beta\in[0,\,1/3].\label{gamma-1}
\end{align}

The main result of this paper is stated as follows.

\begin{theorem}\label{thm:main}
For every $0<\alpha\leq1/4$, define  \begin{align}
R_{\alpha}\triangleq h(\alpha)
-\frac{1}{2}h(2\alpha)
+
(1-\alpha)
\bigg[
h\bigg(
\Gamma^{-1}
\Big(
\frac{\alpha}{1-\alpha}
\Big)
\bigg)
+
1
-
\Gamma^{-1}
\Big(
\frac{\alpha}{1-\alpha}
\Big)
\bigg].\label{def-R-alpha}
 \end{align}
Then, the optimal rate~$R$ for  unique-sum pairs satisfies
\begin{align*}
R \leq
\inf_{0<\alpha\leq1/4} R_{\alpha}.
\end{align*}
\end{theorem}

By Theorem~\ref{thm:main} and a numerical calculation, we obtain
\begin{align*}
R
\leq
\inf_{0<\alpha\leq1/4}R_{\alpha}
\approx
1.4884,
\end{align*}
where the numerical minimum is attained approximately at
$\alpha=0.1437$.
Thus,   the upper bound obtained in
Theorem~\ref{thm:main}   is strictly smaller than Lindstr\"om's  upper bound~$3/2$.
The following remark further shows that the upper bound in Theorem~\ref{thm:main}
is strictly smaller than Lindstr\"om's classical upper bound~$3/2$
without relying on the above numerical calculation.

\begin{remark}
We take  $\alpha=1/4$  and consider
\begin{align}
R_{1/4}
&=
h \Big(\,\frac{1}{4}\,\Big)
-\frac{1}{2}h \Big(\,\frac{1}{2}\,\Big)+
\frac{3}{4}
\bigg[
h\bigg(
\Gamma^{-1}
\Big(\,\frac{1}{3}\,\Big)
\bigg)
+
1
-
\Gamma^{-1}
\Big(\,\frac{1}{3}\,\Big)
\bigg].\label{R-1/4}
\end{align}
It is easy to verify that the function
$h(\lambda)+1-\lambda$
is strictly decreasing on $[1/3,\,1/2]$
(also see the proof of Lemma~\ref{lem:function-properties}
in Section~\ref{sec:system-bound}).
Together with $\Gamma^{-1}
 (1/3)\in(1/3,\,1/2]$ by \eqref{gamma-1}, this implies that
\begin{align}
h\Big(
\Gamma^{-1}
\Big(\,\frac{1}{3}\,\Big)
\Big)
+
1
-
\Gamma^{-1}
\Big(\,\frac{1}{3}\,\Big)<
h\Big(\,
\frac{1}{3}\,
\Big) +1-\frac{1}{3} =
\log 3.
\label{eq:eta0-upper-bound}
\end{align}
Combining \eqref{eq:eta0-upper-bound} with \eqref{R-1/4}, we obtain
\begin{align*}
R_{1/4} <
h\Big(\,\frac{1}{4}\,\Big)
-\frac{1}{2}h\Big(\,\frac{1}{2}\,\Big)+
\frac{3}{4}\log3
 =
\Big(2-\frac{3}{4}\log 3\Big)
-\frac{1}{2}
+\frac{3}{4}\log 3
 =
\frac{3}{2}.
\end{align*}
By Theorem~\ref{thm:main}, we immediately have
\begin{align*}
R \leq
\inf_{0<\alpha\leq1/4} R_{\alpha}\leq R_{1/4} <
\frac{3}{2}.
\end{align*}
This shows that  the upper bound obtained in Theorem~\ref{thm:main} is strictly smaller than
Lindstr\"om's   upper bound~$3/2$.
 \end{remark}

To establish the  upper bound on the optimal rate in
Theorem~\ref{thm:main}, we develop a
 coordinate projection approach that
constructs a lower-dimensional unique-sum system from a unique-sum pair.
To be specific, in Section~\ref{sec:system-bound}, we first introduce a
natural generalization of a unique-sum pair, referred to as a
unique-sum system, and derive a relaxed form of
a necessary condition established by Ordentlich and Shayevitz
on unique-sum systems.
In Section~\ref{sec:code-to-system}, we present the
 coordinate projection approach  for constructing a lower-dimensional unique-sum system
from a unique-sum pair.
Finally, in Section~\ref{sec:analytic-bound}, we establish the upper
bound by combining the results obtained from this approach with  the relaxed form of the  necessary condition on unique-sum systems derived
in Section~\ref{sec:system-bound}.

\section{Unique-Sum Systems}
\label{sec:system-bound}

In this section, we first introduce a natural generalization of a
unique-sum pair, referred to as a \emph{unique-sum system}, and  define the admissible triples for unique-sum systems.
We note that a unique-sum system is equivalent to a
\emph{multiset-union-free system} introduced by Ordentlich and
Shayevitz~\cite{OrdentlichShayevitz2016}.
We then recall a necessary condition established by Ordentlich and Shayevitz \cite{OrdentlichShayevitz2016}
for  admissible triples.
Finally, we derive a relaxed form of this necessary condition, which will be used to establish the upper bound in
Theorem~\ref{thm:main}.

\begin{definition}\label{def-u-s-sys}
Let $n$ and $t$  be two positive integers, and let $\mathcal{S}=\{(\mathcal{X}_{\ell},\mathcal{Y}_\ell)\}_{\ell=1}^t$ be a   collection
such that
$\mathcal{X}_{\ell}$ and $\mathcal{Y}_{\ell}$ are nonempty subsets of $\{0,1\}^n$
for every $1\leq \ell\leq t$.
Then, $\mathcal{S}$  is called an $n$-dimensional unique-sum system if
 \begin{align*}
 |\mathcal{X}_1|=|\mathcal{X}_2|=\cdots=
|\mathcal{X}_t| \quad\textup{ and }\quad
|\mathcal{Y}_1|=|\mathcal{Y}_2|=\cdots=
|\mathcal{Y}_t|;
\end{align*}
and
  the mapping $
(\ell,\bx,\by)
\longmapsto
\bx+\by$
is injective on the set
$$
\big\{(\ell,\bx,\by):\ 1\leq \ell\leq t,\ \bx\in\mbX_\ell,\ \by\in\mbY_\ell\big\}.
$$
\end{definition}

%
%

%

For an $n$-dimensional unique-sum system
$\mathcal{S}=\{(\mathcal X_{\ell},\mathcal Y_{\ell})\}_{\ell=1}^{t}$,
the \emph{rate tuple} of $\mathcal S$, denoted by $\big(r_0(\mathcal S),r_1(\mathcal S),r_2(\mathcal S)\big)$, is defined by
\begin{align*}
r_0(\mathcal S)&\triangleq\frac{1}{n}\log t,\quad
r_1(\mathcal S)\triangleq\frac{1}{n}\log|\mathcal X_1|,
\quad\text{and}\quad
r_2(\mathcal S)\triangleq\frac{1}{n}\log|\mathcal Y_1|,
\end{align*}
 where we note that
$|\mathcal X_1|=|\mathcal X_2|=\cdots=|\mathcal X_t|$ and
$|\mathcal Y_1|=|\mathcal Y_2|=\cdots=|\mathcal Y_t|$ by Definition~\ref{def-u-s-sys}.

%

Furthermore, a triple $(r_0,r_1,r_2)$ of nonnegative real numbers is
called \emph{admissible} for unique-sum systems if, for every
$\epsilon>0$ and every positive integer $N$, there exist an
integer $n\geq N$ and an $n$-dimensional unique-sum system
$\mathcal S_n$ such that
\begin{align*}
r_0&<r_0(\mathcal S_n)+\epsilon,\quad
r_1<r_1(\mathcal S_n)+\epsilon,\quad\text{and}\quad
r_2<r_2(\mathcal S_n)+\epsilon.
\end{align*}
 We remark that the above definition of an admissible triple for
unique-sum systems is equivalent to the definition of an admissible
triple for multiset-union-free systems used in~\cite{OrdentlichShayevitz2016}.

%

In the following, we recall the necessary condition for  admissible triples established in~\cite{OrdentlichShayevitz2016}.
Before that, we introduce some notation.
Let
\begin{align*}
 p\star q\triangleq p(1-q)+q(1-p),\qquad\forall~ p,\,q\in[0,\, 1].
\end{align*}
For any $\alpha\in[0,\, 1]$,  define
\begin{align}
 L(\alpha)\triangleq h(\alpha)+1-\alpha.
 \label{eq:def-L}
\end{align}
Furthermore, for every pair $(p,\alpha)$ satisfying
$
0\leq p\leq\alpha\leq 1/2,$
define
\begin{align}
 J(p,\alpha)
 \triangleq
 \begin{cases}
 \displaystyle
 ~2h \bigg(\frac{1-\sqrt{1-2\alpha}}{2}\bigg)-\alpha,
 &~\textup{if}~p\star p\leq\alpha,\\[4mm]
 \displaystyle
 ~2h \left(
 \frac12\bigg(
 1-\frac{1-\alpha-p\star p}{\sqrt{1-2(p\star p)}}
 \bigg)
 \right)
 -\frac12\left(
 1-\frac{(1-\alpha-p\star p)^2}{1-2(p\star p)}
 \right)
 &~\textup{if}~ p\star p>\alpha.
 \end{cases}
 \label{eq:def-J}
\end{align}

\begin{lemma}[\!\!{\cite[Lemma~3]{OrdentlichShayevitz2016}}]\label{lem:OS}
For each admissible triple $(r_0,r_1,r_2)$,\footnote{We note that every admissible triple $(r_0,r_1,r_2)$ for unique-sum systems satisfies
$0\leq r_1,r_2\leq1$.}
\begin{align*}
 r_0+r_1+r_2
 \leq
 \max_{h^{-1}(r_1)\leq\alpha\leq1/2}
 \min\big\{
 L(\alpha),~
 J\big(h^{-1}(r_1),\alpha\big)+r_0
 \big\},
\end{align*}
where  $h^{-1}$
denotes
the inverse function of the restriction of    $h$ to   $[0,\,1/2]$.
\end{lemma}

Now, we present a relaxed form of
Lemma~\ref{lem:OS} that will be used to establish the upper bound in
Theorem~\ref{thm:main}.

\begin{proposition}\label{cor:Gamma-system-bound}
For every $\beta\in[0,\,1/3]$, each admissible triple
$(r_0,r_1,r_2)$ with
$r_0\leq\beta$
 satisfies
\begin{align*}
r_0+r_1+r_2
\leq
L\big(\Gamma^{-1}(\beta)\big).
\end{align*}
\end{proposition}

\subsection{Proof of Proposition~\ref{cor:Gamma-system-bound}}

In this subsection, we prove Proposition~\ref{cor:Gamma-system-bound}.
We first introduce some notation and lemmas.
Let
\begin{align}
\widehat{J}(\alpha)
\triangleq
2h\left(
\frac{1-\sqrt{1-2\alpha}}{2}
\right)-\alpha,\qquad~\alpha\in[0,\,1/2].
\label{eq-widehat-J}
\end{align}
By the definitions of the functions $\Gamma$ and $L$  (cf.~\eqref{eq:def-h-hat} and \eqref{eq:def-L}),   we can see that
\begin{align*}
\Gamma(\alpha)=L(\alpha)-\widehat{J}(\alpha),\qquad~\alpha\in[0,\,1/2].
\end{align*}

The following lemma gives some properties of the functions $L$, $\widehat J$, and $\Gamma$.

\begin{lemma}\label{lem:function-properties}
The following properties hold:
\begin{itemize}
\item  $L$ is strictly increasing on $[0,\,1/3]$ and strictly decreasing on $[1/3,\,1]$;

\item  $\widehat J$ is strictly increasing on $[0,\,1/2]$;

\item  $
\Gamma =
L-\widehat J$
is strictly decreasing on $[1/3,\,1/2]$ with
$
\Gamma(1/3)>1/3$ and $
\Gamma(1/2)=0.$
\end{itemize}
\end{lemma}

\begin{proof}
First,
for the function $L$ (cf.~\eqref{eq:def-L}), we calculate the first
derivative as
\begin{align*}
L'(\alpha)
=
\log\frac{1-\alpha}{\alpha}-1,\qquad \alpha\in(0,\,1).
\end{align*}
We can see that $L'(\alpha)>0$ for all
$\alpha\in(0,\,1/3)$, and $L'(\alpha)<0$ for all
$\alpha\in(1/3,\,1)$.
Therefore,~$L$ is strictly increasing on $[0,\,1/3]$ and strictly decreasing on $[1/3,\,1]$.

In the following, we consider the function   $\widehat{J}$ (cf.~\eqref{eq-widehat-J}).
Define the following  two functions:
\begin{alignat*}{2}
g_1(\alpha)&\triangleq
\frac{1-\sqrt{1-2\alpha}}{2},&\qquad ~~\alpha&\in[0,\,1/2],\\
g_2(\alpha)&\triangleq 2h(\alpha)-2\alpha(1-\alpha),&\qquad \alpha&\in[0,\,1/2].
\end{alignat*}
Together with the definition of $\widehat J$, we can see that
\begin{align*}
\widehat J(\alpha)=
g_2\big(g_1(\alpha)\big),
\qquad
\alpha\in[0,\,1/2].
\end{align*}
To show that $\widehat J$ is strictly increasing on $[0,\,1/2]$, it suffices to prove that both $g_1$ and $g_2$
are strictly increasing on $[0,\,1/2]$.
First,
it is easy to see that   $g_1$ is strictly increasing on $[0,\,1/2]$.
We next prove that the function $g_2$ is strictly increasing on
$[0,\,1/2]$.
To be specific, we calculate the first derivative as
\begin{align*}
g_2'(\alpha)
=
2\log\frac{1-\alpha}{\alpha}
-2+4\alpha,
\qquad \alpha\in(0,\,1/2).
\end{align*}
To show that $g_2'(\alpha)>0$ for every
$\alpha\in(0,\,1/2)$, we further calculate the second derivative as
\begin{align*}
g_2''(\alpha)
&=
4-\frac{2}{\ln 2}
\left(
\frac{1}{\alpha}
+
\frac{1}{1-\alpha}
\right)
<0,
\qquad \alpha\in(0,\,1/2).
\end{align*}
This immediately implies that
$g_2'(\alpha)>g_2'(1/2)=0$ for every
$\alpha\in(0,\,1/2)$.
Hence, the function~$g_2$ is strictly increasing on
$[0,\,1/2]$.

Finally,
since $L$ is strictly decreasing on $[1/3,\,1/2]$ and
$\widehat J$ is strictly increasing on $[1/3,\,1/2]$, the function
$
\Gamma =
L-\widehat J$  is strictly decreasing on $[1/3,\,1/2]$.
 Furthermore, by a direct calculation, we have
\begin{align*}
L(1/3)-\widehat{J}(1/3)
&=
\log3+\frac13
-
2h\bigg(
\frac{1-1/\sqrt3}{2}
\bigg)\approx 0.4303
>
\frac13.\\
L(1/2)-\widehat{J}(1/2)
&=0.
\end{align*}
 The lemma is thus proved.
\end{proof}

\begin{lemma}\label{lem:J0-dominates}
For any $0\leq p\leq\alpha\leq1/2$,
\begin{align*}
J(p,\alpha)
\leq  \widehat{J}(\alpha).
\end{align*}
In particular, $J(p,\alpha)
= \widehat{J}(\alpha)$ whenever $p=0$ and $\alpha\in[0,\,1/2]$.
\end{lemma}

\begin{proof}
For the case $p\star p\leq\alpha$,
it follows from the definitions of
$J(p,\alpha)$ and $\widehat{J}(\alpha)$ (cf.~\eqref{eq:def-J} and~\eqref{eq-widehat-J}) that $J(p,\alpha)
= \widehat{J}(\alpha)$. In particular,  $J(p,\alpha)
= \widehat{J}(\alpha)$  whenever $p=0$ and $\alpha\in[0,\,1/2]$, because $p\star p=0\leq\alpha$.

It remains to prove  $J(p,\alpha)
\leq  \widehat{J}(\alpha)$ for  the case $ p\star p>\alpha.$
Define the following function:
\begin{align*}
f(z)
\triangleq
2h \bigg(\frac{1-z}{2}\bigg)
-\frac{1}{2}(1-z^2),
\qquad  z\in [0,\,1].
\end{align*}
For notational simplicity,  let
\begin{align}
z_0\triangleq\sqrt{1-2\alpha}\quad\text{ and }\quad
z_1\triangleq\frac{1-\alpha-(p\star p)}{\sqrt{1-2(p\star p)}},\label{def-z-z0}
\end{align}
where it will be shown below that $0\leq z_0\leq z_1\leq 1$.
Note that $ p\star p>\alpha$.
Together with the definitions of $J$ and $\widehat{J}$ in \eqref{eq:def-J} and \eqref{eq-widehat-J},
we can verify that
\begin{align}
\widehat{J}(\alpha)=f(z_0)\quad\text{ and }\quad J(p,\alpha)=f(z_1).\label{eq-equ-form}
\end{align}
To prove $J(p,\alpha)
\leq  \widehat{J}(\alpha)$, by \eqref{eq-equ-form} it is equivalent to proving $f(z_1)\leq f(z_0)$.

We first prove that
the function $f$ is strictly decreasing on $[0,\,1]$.
To be specific,
we calculate the first
derivative as
\begin{align*}
f'(z)
=
z-\log\frac{1+z}{1-z},
\qquad  z\in (0,\,1).
\end{align*}
To show that $f'(z)<0$ for every $z\in(0,\,1)$, we further calculate the second
derivative as
\begin{align*}
f''(z)
&=
1-\frac{1}{\ln 2}
\left(
\frac{1}{1+z}
+
\frac{1}{1-z}
\right) =1-
\frac{2}{(1-z^2)\ln 2} <0,
\qquad z\in(0,\,1).
\end{align*}
This immediately implies that $f'(z)<f'(0)=0$ for every $z\in(0,\,1)$. Hence, the function $f$ is strictly decreasing on $[0,\,1]$.

Next, we prove that
\[
0\leq z_0\leq z_1\leq 1,
\]
where we recall that $z_0$ and $z_1$ are defined in
\eqref{def-z-z0}.
It first follows from $\alpha\leq 1/2$ that
$z_0\geq 0$.
Since $p\star p>\alpha$ and
$0\leq p\leq\alpha\leq1/2$, we have
$
p\star p<1/2$ and $
\alpha<1/2.$
This  implies that  $z_1>0$.
Furthermore, we consider
\begin{align*}
z_0^2-z_1^2
&=
\frac{(1-2\alpha)\big(1-2(p\star p)\big)-\big(1-\alpha-(p\star p)\big)^2}{1-2(p\star p)}
=
-\frac{\big((p\star p)-\alpha\big)^2}
{1-2(p\star p)}
\leq0,
\end{align*}
where we note  that  $
1-2(p\star p)
=
(1-2p)^2
>
0,$
because the condition
$p\star p>\alpha$ together with
$0\leq p\leq\alpha\leq1/2$
implies that $p<1/2$.
This implies that   $
z_0\leq z_1$.
Since $0\leq p\leq\alpha\leq 1/2$, we also have
\begin{align*}
z_1=
\frac{1-\alpha-(p\star p)}
{\sqrt{1-2(p\star p)}}\leq
\frac{1-p-(p\star p)}
{\sqrt{1-2(p\star p)}}=
\frac{(1-p)(1-2p)}
{1-2p}=
1-p
\leq1,
\end{align*}
where the second equality follows from
$1-p-(p\star p)=1-p-2p(1-p)=(1-p)(1-2p)$ and
\begin{align*}
\sqrt{1-2(p\star p)}=
\sqrt{1-4p(1-p)}=\sqrt{(1-2p)^2}=1-2p.
\end{align*}

Based on the above,   the function $f$ is strictly decreasing on $[0,\,1]$, and $0\leq z_0\leq z_1\leq 1.$ This immediately implies that $f(z_1)\leq
f(z_0)$.
The lemma is thus proved.
\end{proof}

%
%
%

\begin{lemma}\label{lem:Gamma-characterization}
For any   $\beta\in[0,\,1/3]$,
\begin{align*}
\max_{0\leq\alpha\leq1/2}
\min\big\{
L(\alpha),~
\widehat{J}(\alpha)+\beta
\big\}=
L\big(\Gamma^{-1}(\beta)\big).
\end{align*}
\end{lemma}

\begin{proof}
Fix an arbitrary $\beta\in[0,\,1/3]$.
Since both the functions $L$ and $\widehat J$ are strictly increasing on
$[0,\,1/3]$ (cf.~Lemma~\ref{lem:function-properties}), the function
$
\min\big\{
L(\alpha),\,
\widehat J(\alpha)+\beta
\big\}$
is also increasing on $[0,\,1/3]$. This implies that
\begin{align}
\max_{0\leq\alpha\leq1/2}
\min\big\{
L(\alpha),~
\widehat{J}(\alpha)+\beta
\big\}=\max_{1/3\leq\alpha\leq1/2}
\min\big\{
L(\alpha),~
\widehat{J}(\alpha)+\beta
\big\}.\label{le-eq}
\end{align}

In the following, we consider the   function
$
\min\big\{
L(\alpha),\,
\widehat J(\alpha)+\beta
\big\}$ on $[1/3,\,1/2]$.
It follows from Lemma~\ref{lem:function-properties} that $
\Gamma$
is strictly decreasing on $[1/3,\,1/2]$ with
$
\Gamma(1/3)>1/3$ and $
\Gamma(1/2)=0.$
Since $\beta\in[0,\,1/3]$, we have  $\Gamma^{-1}(\beta)\in [1/3,\,1/2\big]$. By the monotonicity of $\Gamma$,
we thus obtain
\begin{align*} \Gamma(\alpha)
\geq\Gamma\big(\Gamma^{-1}(\beta)\big)=
\beta
\qquad~~&\textup{if}~
\alpha\in
\big[1/3,~\Gamma^{-1}(\beta)\big];\\[2mm]
~\Gamma(\alpha)
\leq
\Gamma\big(\Gamma^{-1}(\beta)\big)=
\beta
\qquad~~&\textup{if}~
\alpha\in
\big[\Gamma^{-1}(\beta),~1/2\big].
\end{align*}
This, together with   $\Gamma=L-\widehat J$, further implies that
\begin{align}
~L(\alpha)
 \geq
\widehat J(\alpha)+\beta
\qquad~~&\textup{if}~
\alpha\in
\big[1/3,~\Gamma^{-1}(\beta)\big]; \label{case-eq-1}\\[2mm]
~L(\alpha)
 \leq
\widehat J(\alpha)+\beta
\qquad~~&\textup{if}~
\alpha\in
\big[\Gamma^{-1}(\beta),~1/2\big].  \label{case-eq-2}
\end{align}
We can see from \eqref{case-eq-1} and \eqref{case-eq-2} that
\begin{align*}
L\big(\Gamma^{-1}(\beta)\big)
=
\widehat J\big(\Gamma^{-1}(\beta)\big)+\beta.
\end{align*}
By \eqref{case-eq-1} and \eqref{case-eq-2}, the function
$\min\big\{L(\alpha),\,\widehat J(\alpha)+\beta\big\}$
can be written as
\begin{align}
\min\big\{
L(\alpha),\,
\widehat J(\alpha)+\beta
\big\}=\begin{cases}
\widehat J(\alpha)+\beta &\quad\textup{if}~\alpha\in
\big[1/3,~\Gamma^{-1}(\beta)\big);\\[2mm]
L(\alpha) &\quad\textup{if}~\alpha\in
\big[\Gamma^{-1}(\beta),~1/2\big].
\end{cases}\label{equiv-form}
\end{align}
We note that  $\widehat J$ is strictly increasing on
$[0,\,1/2]$ and $L$ is strictly decreasing on
$[1/3,\,1]$ (cf.~Lemma~\ref{lem:function-properties}).
By \eqref{equiv-form}, we can see that
 the function $
\min\big\{
L(\alpha),~
\widehat J(\alpha)+\beta
\big\}$
is strictly increasing on
$\big[1/3,\,\Gamma^{-1}(\beta)\big)$
and strictly decreasing on
$\big[\Gamma^{-1}(\beta),\,1/2\big]$.
Hence,   we obtain
\begin{align*}
\max_{1/3\leq\alpha\leq1/2}
\min\big\{
L(\alpha),~
\widehat{J}(\alpha)+\beta
\big\}
=
L\big(\Gamma^{-1}(\beta)\big).
\end{align*}
This, together with \eqref{le-eq}, completes the proof of the lemma.
\end{proof}

We now prove Proposition~\ref{cor:Gamma-system-bound}.

\begin{proof}[Proof of Proposition~\ref{cor:Gamma-system-bound}]
Fix an arbitrary $\beta\in[0,\,1/3]$ and an admissible triple
$(r_0,r_1,r_2)$ satisfying $r_0\leq\beta$.
By Lemma~\ref{lem:OS}, we first have
\begin{align}
r_0+r_1+r_2
&\leq
\max_{h^{-1}(r_1)\leq\alpha\leq1/2}
\min\big\{
L(\alpha),~
J\big(h^{-1}(r_1),\alpha\big)+r_0
\big\}.
\label{eq:pf-Gamma-system-bound-1}
\end{align}
For every $\alpha$ satisfying  $0\leq h^{-1}(r_1)\leq\alpha\leq1/2$, it follows from
Lemma~\ref{lem:J0-dominates}   that
\begin{align*}
J\big(h^{-1}(r_1),\alpha\big)
\leq
\widehat J(\alpha).
\end{align*}
Together with \eqref{eq:pf-Gamma-system-bound-1}, we further have
\begin{align}
r_0+r_1+r_2
&\leq
\max_{h^{-1}(r_1)\leq\alpha\leq1/2}
\min\big\{
L(\alpha),~
\widehat J(\alpha)+r_0
\big\}\nonumber\\
&\leq
\max_{0\leq\alpha\leq1/2}
\min\big\{
L(\alpha),~
\widehat J(\alpha)+r_0
\big\}\nonumber\\
&\leq
\max_{0\leq\alpha\leq1/2}
\min\big\{
L(\alpha),~
\widehat J(\alpha)+\beta
\big\} \label{eq:pf-Gamma-system-bound-3}\\
&=
L\big(\Gamma^{-1}(\beta)\big),\label{eq:pf-Gamma-system-bound-4}
\end{align}
where the inequality \eqref{eq:pf-Gamma-system-bound-3} follows from  $r_0\leq\beta$, and the  equality  \eqref{eq:pf-Gamma-system-bound-4} follows from Lemma~\ref{lem:Gamma-characterization}.
The proposition is thus proved.
\end{proof}

%

%

\section{From a Unique-Sum Pair to a Unique-Sum System}
\label{sec:code-to-system}

In this section, we present a coordinate projection approach for
obtaining a lower-dimensional unique-sum system from a unique-sum pair.
By applying this approach, we prove that there exists an admissible
triple for unique-sum systems that satisfies a condition involving the
optimal rate for unique-sum pairs.

%

For a binary vector
$\boldsymbol{x}\in\{0,1\}^{n}$,
$\mathrm{wt}(\boldsymbol{x})$ denotes the Hamming weight of
$\boldsymbol{x}$, i.e., the number of nonzero components of
$\boldsymbol{x}$.
Furthermore,
$\boldsymbol{x}\oplus\boldsymbol{y}$ denotes the componentwise
modulo-two sum of two binary vectors
$\boldsymbol{x}$ and $\boldsymbol{y}$ in $\{0,1\}^{n}$.

\begin{lemma}\label{lem:mean-xor-at-least-half}
For each $n$-dimensional unique-sum pair $(\mbX,\mbY)$, there exists an
$n$-dimensional unique-sum pair $(\mbX',\mbY)$
such that
\begin{align*}
|\mbX'|=|\mbX|\quad\textup{ and }\quad
\E\big[\wt(\bX'\oplus\bY)\big]\geq\frac n2,
\end{align*}
where $\bX'$ and~$\bY$ are  two independent random vectors uniformly
distributed over $\mbX'$ and~$\mbY$, respectively.
\end{lemma}

\begin{proof}
First,
let $\bX$ and~$\bY$ be  two independent random vectors uniformly distributed over
$\mathcal{X}$ and $\mathcal{Y}$, respectively.
We consider two cases below.

\medskip
\noindent
\textbf{Case 1:} $\E\big[\wt(\bX\oplus\bY)\big]\geq n/2$.
\medskip

For this case,  we take
$
\mbX'=\mbX
$ and  $\boldsymbol{X}'= \boldsymbol{X}$.
Then, we have   $$\E\big[\wt(\bX'\oplus\bY)\big]=\E\big[\wt(\bX\oplus\bY)\big]\geq \frac{n}{2}.$$

\noindent
\textbf{Case 2:} $\E\big[\wt(\bX\oplus\bY)\big]< n/2$.
\medskip

For this case,
we take
$
\mbX'=
\{\one-\bx:\ \bx\in\mbX\}
$ and  $\boldsymbol{X}'= \boldsymbol{1}-\boldsymbol{X}$.\footnote{Throughout this paper, $\boldsymbol{1}$ denotes the all-ones
vector of appropriate dimension.}
 Clearly,
$
|\mbX'|=|\mbX|
$, and $\boldsymbol{X}'$ and~$\boldsymbol{Y}$ are two independent random vectors uniformly distributed over
$\mathcal{X}'$ and $\mathcal{Y}$, respectively.

We claim that $( \mbX',\mbY)$ is   a    unique-sum pair.
Otherwise, we assume that
$(\mbX',\mbY)$ is not a unique-sum pair, namely that there exist two distinct pairs
$(\one-\bx_1,\by_1)$ and $(\one-\bx_2,\by_2)$
in $\mbX'\times\mbY$ with
$\bx_1,\bx_2\in\mbX$ and $\by_1,\by_2\in\mbY$
such that
\begin{align*}
(\one-\bx_1)+\by_1
=
(\one-\bx_2)+\by_2
\quad\text{i.e.,}\quad
\bx_2+\by_1
=
\bx_1+\by_2 .
\end{align*}
We note that  $(\mbX,\mbY)$ is a unique-sum pair, and  $(\bx_2,\by_1),
(\bx_1,\by_2)\in\mbX\times\mbY
$. Then, the equality
$
\bx_2+\by_1
=
\bx_1+\by_2
$
implies that
$\bx_2=
\bx_1$ and $
\by_1=\by_2
$.
This contradicts the assumption that
$(\boldsymbol 1-\boldsymbol x_1,\boldsymbol y_1)$ and
$(\boldsymbol 1-\boldsymbol x_2,\boldsymbol y_2)$ are distinct.

Furthermore, we have
\begin{align*}
\E\big[\wt(\bX'\oplus\bY)\big]
&=
\E\big[\wt\big((\one-\bX)\oplus\bY\big)\big]\\
&=
\E\big[n-\wt(\bX\oplus\bY)\big]\\
&=
n-\E\big[\wt(\bX\oplus\bY)\big]\\
&>
\frac n2,
\end{align*}
where the inequality follows from  $
\E\big[\wt(\bX\oplus\bY)\big]< n/2.
$

Combining the above two cases, we have proved the lemma.
\end{proof}

In the remainder of this section, let
$n$ be an arbitrary positive integer, and let
$(\mathcal X,\mathcal Y)$ be an arbitrary $n$-dimensional unique-sum pair.
Further, let
$\boldsymbol{X}$ and $\boldsymbol{Y}$ be two independent random vectors
uniformly distributed over $\mathcal{X}$ and $\mathcal{Y}$, respectively.
By Lemma~\ref{lem:mean-xor-at-least-half}, we assume without loss of generality that
\begin{align*}
\E\big[\wt(\bX\oplus\bY)\big]\geq\frac n2.
\end{align*}

For every $1\leq i\leq n$,  we use $\boldsymbol{x}[i]$ to denote  the $i$-th component of
a binary vector $\boldsymbol{x}\in\{0,1\}^n$.
Furthermore, for every subset
$I\subseteq[n]\triangleq\{1,2,\cdots,n\}$, we let $\boldsymbol{x}[I]
\triangleq
(\boldsymbol{x}[i]:~i\in I)$,
i.e., the restriction of $\boldsymbol{x}$ to the coordinates in
$I$.
For every subset $I\subseteq[n]$ and  every $\boldsymbol{v}\in\{0,1\}^{|I|}$, we define the following two subsets of
$\{0,1\}^{n-|I|}$:
\begin{align}
\mathcal{X}_{I,\boldsymbol{v}}
&\triangleq
\big\{\boldsymbol{x}[I^{\mathrm{c}}]:~
\boldsymbol{x}\in\mathcal{X}~\textup{such that }~
\boldsymbol{x}[I]=\boldsymbol{v}\big\},
\label{eq:def-An-t}\\[1mm]
\mathcal{Y}_{I,\boldsymbol{1}-\boldsymbol{v}}
&\triangleq
\big\{\boldsymbol{y}[I^{\mathrm{c}}]:~
\boldsymbol{y}\in\mathcal{Y}~\textup{such that}~
\boldsymbol{y}[I]=\boldsymbol{1}-\boldsymbol{v}\big\},
\label{eq:def-Bn-t}
\end{align}
where
$I^{\mathrm{c}}$
denotes the complement of $I$ in $[n]$, i.e., $I^{\mathrm{c}}\triangleq[n]\setminus I$.
For   every subset $I\subseteq[n]$  and every pair $(i,j)$ with $0\leq i,j\leq n-|I|$, define
\begin{align}
\mathcal{V}_{I}^{(i,j)}
\triangleq\Big\{\boldsymbol{v}\in\{0,1\}^{|I|}:~
 2^i\leq |
\mathcal X_{ I,\boldsymbol{v}} |<2^{i+1},~2^j\leq |\mathcal Y_{ I,\boldsymbol 1-\boldsymbol{v}} |<2^{j+1}
\Big\}.
\label{eq:def-Tnij}
\end{align}

We next give a crucial lemma on
$\mathcal{V}_{I}^{(i,j)}$, whose proof is deferred to
Section~\ref{subsec-pf-lemma}.

\begin{lemma}\label{lemma:coordinate-dyadic-selection}
For every integer $0\leq k\leq\lfloor n/2\rfloor$, there exist a subset
$I\subseteq[n]$ with $|I|=k$ and two nonnegative integers
$i,j$ with
$0\leq i,j\leq n-|I|$
such that
\begin{align*}
\big|\mathcal{V}_{I}^{(i,j)}\big|\cdot 2^{i+j}
\geq
|\mbX||\mbY|\cdot
\frac{\binom{\lfloor n/2\rfloor}{k}}
{\binom nk}\cdot
\frac{1}{4(n-k+1)^2}.
\end{align*}
\end{lemma}

Consider a subset $I\subseteq[n]$  and a pair $(i,j)$ with $0\leq i,j\leq n-|I|$.
By \eqref{eq:def-Tnij}, we can see that
  $ |
\mathcal X_{ I,\boldsymbol{v}} |\geq 2^i $ and $|\mathcal Y_{ I,\boldsymbol 1-\boldsymbol{v}} |\geq 2^j$
for every $\boldsymbol{v}\in \mathcal{V}_{I}^{(i,j)}.$
As such, for every $\boldsymbol{v}\in \mathcal{V}_{I}^{(i,j)},$ we
 use
$\mathcal X_{I,\boldsymbol{v}}^{(i)}$ and
$\mathcal Y_{I,\boldsymbol 1-\boldsymbol{v}}^{(j)}$
to denote arbitrary subsets of
$\mathcal X_{I,\boldsymbol{v}}$ and
$\mathcal Y_{I,\boldsymbol 1-\boldsymbol{v}}$, respectively, satisfying
\begin{align}
\big|\mathcal X_{I,\boldsymbol{v}}^{(i)}\big|
=
2^i
\quad\text{ and }\quad
\big|\mathcal Y_{I,\boldsymbol 1-\boldsymbol{v}}^{(j)}\big|
=
2^j.
\label{eq:def-dyadic-subsets}
\end{align}

With the above notation, the following lemma presents the construction of a
unique-sum system from the unique-sum pair $(\mathcal X,\mathcal Y)$.

\begin{lemma}\label{lem:section-injectivity}
For  any   $I\subseteq[n]$ and $0\leq i,j\leq n-|I|$ with $\mathcal{V}_{I}^{(i,j)}\neq\emptyset$,
the collection
$$\Big\{
\big(
\mbX_{I,\boldsymbol{v}}^{(i)},
\mbY_{I,\boldsymbol{1}-\boldsymbol{v}}^{(j)}
\big):~
\boldsymbol{v}\in\mathcal{V}_{I}^{(i,j)}
\Big\}$$ forms  an $(n-|I|)$-dimensional unique-sum system.
\end{lemma}

\begin{proof}
For every $\boldsymbol{v}\in \mathcal{V}_{I}^{(i,j)}$,
it follows from the definitions of $\mathcal X_{I,\boldsymbol{v}}, \mathcal Y_{I,\boldsymbol{1}-\boldsymbol{v}},\mathcal X_{I,\boldsymbol{v}}^{(i)},\mathcal Y_{I,\boldsymbol{1}-\boldsymbol{v}}^{(j)}$ (cf.~\eqref{eq:def-An-t}, \eqref{eq:def-Bn-t} and \eqref{eq:def-dyadic-subsets})
  that
\begin{align*}
\mathcal X_{I,\boldsymbol{v}}^{(i)}\subseteq \mathcal X_{I,\boldsymbol{v}}\subseteq \{0,1\}^{n-|I|},\qquad
\mathcal Y_{I,\boldsymbol{1}-\boldsymbol{v}}^{(j)}\subseteq \mathcal Y_{I,\boldsymbol{1}-\boldsymbol{v}}\subseteq \{0,1\}^{n-|I|},
\end{align*}
and
\begin{align*}
\big|\mbX_{I,\boldsymbol{v}}^{(i)}\big|=2^i,\qquad
\big|\mbY_{I,\boldsymbol1-\boldsymbol{v}}^{(j)}\big|=2^j.
\end{align*}
Therefore, it remains to prove that the mapping
$
(\boldsymbol{v},\boldsymbol{a},\boldsymbol{b})
\longmapsto
\boldsymbol{a}+\boldsymbol{b}$
is injective on
the set
\begin{align*}
\Big\{(\boldsymbol{v},\boldsymbol{a},\boldsymbol{b}):~
\boldsymbol{v}\in\mathcal{V}_{I}^{(i,j)},\,
\boldsymbol{a}\in \mathcal X_{I,\boldsymbol{v}}^{(i)},\,
\boldsymbol{b}\in\mathcal Y_{I,\boldsymbol{1}-\boldsymbol{v}}^{(j)}\Big\}.
\end{align*}

Since $\mathcal{V}_{I}^{(i,j)} \subseteq\{0,1\}^{|I|}$,
$
\mathcal X_{I,\boldsymbol{v}}^{(i)}
\subseteq
\mathcal X_{I,\boldsymbol{v}},$ and $
\mathcal Y_{I,\boldsymbol{1}-\boldsymbol{v}}^{(j)}
\subseteq
\mathcal Y_{I,\boldsymbol{1}-\boldsymbol{v}},
$
it suffices to prove that the mapping
$
(\boldsymbol{v},\boldsymbol{a},\boldsymbol{b})
\longmapsto
\boldsymbol{a}+\boldsymbol{b}$
is injective on
the set
\begin{align*}
\Big\{(\boldsymbol{v},\boldsymbol{a},\boldsymbol{b}):~
\boldsymbol{v}\in\{0,1\}^{|I|},\,
\boldsymbol{a}\in \mathcal{X}_{I, \boldsymbol{v}} ,\,
\boldsymbol{b}\in \mathcal{Y}_{I,\boldsymbol{1}-\boldsymbol{v}}\Big\}.
\end{align*}
Otherwise, we assume that there exist two distinct tuples
$(\boldsymbol{v},\boldsymbol{a},\boldsymbol{b})$ and
$(\boldsymbol{v}',\boldsymbol{a}',\boldsymbol{b}')$  with $\boldsymbol{v}\in\{0,1\}^{|I|},
\boldsymbol{a}\in \mathcal{X}_{I, \boldsymbol{v}} ,
\boldsymbol{b} \in \mathcal{Y}_{I,\boldsymbol{1}-\boldsymbol{v}}$
and $\boldsymbol{v}'\in\{0,1\}^{|I|},
\boldsymbol{a}'\in \mathcal{X}_{I, \boldsymbol{v}'},
 \boldsymbol{b}'\in \mathcal{Y}_{I,\boldsymbol{1}-\boldsymbol{v}'}$
 such that $\boldsymbol{a}+\boldsymbol{b}=
\boldsymbol{a}'+\boldsymbol{b}'.$
Since $\boldsymbol{a}\in \mathcal{X}_{I, \boldsymbol{v}}$ and $
\boldsymbol{b} \in \mathcal{Y}_{I,\boldsymbol{1}-\boldsymbol{v}}$,
there exist
$\boldsymbol{x}\in\mathcal{X}$ and
$\boldsymbol{y}\in\mathcal{Y}$ such that
\begin{align}
\boldsymbol{x}[I]&=
\boldsymbol{v},\qquad
\boldsymbol{x} [I^{\mathrm{c}}]=
\boldsymbol{a},\qquad
\boldsymbol{y}\big[I\big]=
\boldsymbol{1}-\boldsymbol{v},\qquad
\boldsymbol{y}[I^{\mathrm{c}}]=
\boldsymbol{b}.
\label{pf-section-injectivity-eq2}
\end{align}
Similarly, since $\boldsymbol{a}'\in \mathcal{X}_{I, \boldsymbol{v}'}$ and $
 \boldsymbol{b}'\in \mathcal{Y}_{I,\boldsymbol{1}-\boldsymbol{v}'}$, there exist
$\boldsymbol{x}'\in\mathcal{X}$ and
$\boldsymbol{y}'\in\mathcal{Y}$ satisfying
\begin{align}
\boldsymbol{x}'[I]=
\boldsymbol{v}',\qquad
\boldsymbol{x}' [I^{\mathrm{c}}]=
\boldsymbol{a}',\qquad
\boldsymbol{y}'\big[I\big]=
\boldsymbol{1}-\boldsymbol{v}',\qquad
\boldsymbol{y}' [I^{\mathrm{c}}]=
\boldsymbol{b}'.
\label{pf-section-injectivity-eq3}
\end{align}
By \eqref{pf-section-injectivity-eq2} and \eqref{pf-section-injectivity-eq3}, we have
\begin{align}
(\boldsymbol{x}+\boldsymbol{y})[I]&=\boldsymbol{x}[I]+\boldsymbol{y}[I]\nonumber\\
&=
\boldsymbol{v}
+
\boldsymbol{1}-\boldsymbol{v}\nonumber\\
&=\boldsymbol{1}\nonumber\\
&=\boldsymbol{v}'
+
(\boldsymbol{1}-\boldsymbol{v}')\nonumber\\
&=\boldsymbol{x}'[I]+\boldsymbol{y}'[I]=(\boldsymbol{x}'+\boldsymbol{y}')[I].
\label{pf-section-injectivity-eq4}
\end{align}
By  \eqref{pf-section-injectivity-eq2}, \eqref{pf-section-injectivity-eq3}, and  $\boldsymbol{a}+\boldsymbol{b}=
\boldsymbol{a}'+\boldsymbol{b}'$, we also have
\begin{align}
(\boldsymbol{x}+\boldsymbol{y}) [I^{\mathrm{c}}]&=\boldsymbol{x} [I^{\mathrm{c}}]+\boldsymbol{y} [I^{\mathrm{c}}]
\nonumber\\
&=
\boldsymbol{a}+\boldsymbol{b}
\nonumber\\
&=
\boldsymbol{a}'+\boldsymbol{b}'
\nonumber\\
&=\boldsymbol{x}' [I^{\mathrm{c}}]+\boldsymbol{y}' [I^{\mathrm{c}}]=
(\boldsymbol{x}'+\boldsymbol{y}') [I^{\mathrm{c}}].
\label{pf-section-injectivity-eq6}
\end{align}
It then follows from \eqref{pf-section-injectivity-eq4} and \eqref{pf-section-injectivity-eq6} that
$\boldsymbol{x}+\boldsymbol{y}=
\boldsymbol{x}'+\boldsymbol{y}'.$
Since
$(\mathcal{X},\mathcal{Y})$ is an $n$-dimensional unique-sum pair,  $\boldsymbol{x}+\boldsymbol{y}=
\boldsymbol{x}'+\boldsymbol{y}'$ implies that
$
\boldsymbol{x}=
\boldsymbol{x}'$ and $
\boldsymbol{y}=
\boldsymbol{y}'.
$
Together with \eqref{pf-section-injectivity-eq2} and~\eqref{pf-section-injectivity-eq3},  we further have $\boldsymbol{v}=
\boldsymbol{v}',
\boldsymbol{a}=
\boldsymbol{a}',
\boldsymbol{b}=
\boldsymbol{b}'.$
This contradicts the assumption that
$(\boldsymbol v,\boldsymbol a,\boldsymbol b)$ and
$(\boldsymbol v',\boldsymbol a',\boldsymbol b')$
are distinct.
The lemma is thus proved.
\end{proof}

\begin{proposition}\label{prop:extracted-triple}
For every $0<\alpha\leq1/4$, there exists an admissible triple
$(r_0,r_1,r_2)$ for unique-sum systems satisfying
\begin{align*}
r_0
\leq
\frac{\alpha}{1-\alpha}\quad\textup{ and }\quad
r_0+r_1+r_2
\geq
\frac{1}{1-\alpha}\bigg(R+\frac12h(2\alpha)-h(\alpha)\bigg).
\end{align*}
\end{proposition}

\begin{proof}[Proof of Proposition~\ref{prop:extracted-triple}]
Fix an arbitrary real number $0<\alpha\leq1/4$.
Recall that $n$ is an arbitrary positive integer and
$(\mathcal X,\mathcal Y)$ is an $n$-dimensional unique-sum pair.
We further take $(\mathcal X,\mathcal Y)$ such that
\begin{align}
|\mathcal X||\mathcal Y|=M(n),
\label{eq-M}
\end{align}
where we recall \eqref{eq:def-Mn} for the definition of $M(n)$.
Since $0<\alpha\leq1/4$, we have
$\lfloor\alpha n\rfloor\leq\lfloor n/2\rfloor$.
For the above $n$-dimensional unique-sum pair
$(\mathcal X,\mathcal Y)$, by
Lemma~\ref{lemma:coordinate-dyadic-selection}, we choose a subset
$I\subseteq[n]$ with $|I|=\lfloor\alpha n\rfloor$ and a pair $(i,j)$
with $0\leq i,j\leq n-|I|$ such that
\begin{align}
\big|\mathcal{V}_{I}^{(i,j)}\big|\cdot 2^{i+j}
\geq
|\mbX||\mbY|\cdot
\frac{~\binom{\lfloor n/2\rfloor}{\lfloor\alpha n\rfloor}~}
{\binom{n}{\lfloor\alpha n\rfloor}}
\cdot
\frac{1}{4(n-\lfloor\alpha n\rfloor+1)^2}.
\label{eq:pf-extracted-triple-main}
\end{align}
Since the right-hand side of
\eqref{eq:pf-extracted-triple-main} is strictly positive, we have
$
\big|\mathcal V_I^{(i,j)}\big|>0,$
and thus $\mathcal V_I^{(i,j)}\neq\emptyset$.
It further follows from Lemma~\ref{lem:section-injectivity} that
the collection
$$\Big\{
\big(
\mbX_{I,\boldsymbol{v}}^{(i)},
\mbY_{I,\boldsymbol{1}-\boldsymbol{v}}^{(j)}
\big):~
\boldsymbol{v}\in\mathcal{V}_{I}^{(i,j)}
\Big\}$$ forms  an $(n-|I|)$-dimensional unique-sum system. Note that $|I|=\lfloor\alpha n\rfloor$,  and  $
\big|\mathcal X_{I,\boldsymbol{v}}^{(i)}\big|
=
2^i$ and $
\big|\mathcal Y_{I,\boldsymbol 1-\boldsymbol{v}}^{(j)}\big|
=
2^j$ for all $\boldsymbol{v}\in\mathcal{V}_{I}^{(i,j)}$ (cf.~\eqref{eq:def-dyadic-subsets}).
We denote the rate tuple of the above system by
$(r_{0,n},r_{1,n},r_{2,n})$, where
\begin{align}
r_{0,n}
&\triangleq  \frac{1}{n-\lfloor\alpha n\rfloor}\log\big|\mathcal{V}_{I}^{(i,j)}\big|,
\label{eq:def-r0n}\\
r_{1,n}
&\triangleq\frac{1}{n-\lfloor\alpha n\rfloor}\log  2^{i}= \frac{i}{n-\lfloor\alpha n\rfloor},
\label{eq:def-r1n}\\
r_{2,n}
&\triangleq\frac{1}{n-\lfloor\alpha n\rfloor}\log   2^{j} =\frac{j}{n-\lfloor\alpha n\rfloor}.
\label{eq:def-r2n}
\end{align}

Before discussing further, we need the following two inequalities:
\begin{align}
\limsup_{n\to\infty}r_{0,n}
&\leq
\frac{\alpha}{1-\alpha},
\label{eq:pf-r0-epsilon}\\
\liminf_{n\to\infty}
\big(r_{0,n}+r_{1,n}+r_{2,n}\big)
&\geq
\frac{1}{1-\alpha}
\bigg(
R+\frac12h(2\alpha)-h(\alpha)
\bigg),
\label{eq:pf-total-epsilon}
\end{align}
which will be verified at the end of the proof.
We note that the sequence
$\{(r_{0,n},r_{1,n},r_{2,n})\}_{n\geq1}$  is bounded.
To be specific,  it follows from
$\mathcal V_I^{(i,j)}\subseteq\{0,1\}^{|I|}$
and $|I|=\lfloor\alpha n\rfloor$ that
\begin{align}
0
\leq
r_{0,n}
\leq \frac{1}{n-\lfloor\alpha n\rfloor}\log 2^{|I|}
=
\frac{\lfloor\alpha n\rfloor}
{n-\lfloor\alpha n\rfloor}
\leq
\frac{\alpha}{1-\alpha}
\leq
\frac13,\label{bounded-eq1}
\end{align}
where the last two inequalities follow from
$\lfloor\alpha n\rfloor\leq\alpha n$
and $0<\alpha\leq1/4$, respectively.
Furthermore, by
$0\leq i,j\leq n-|I|$ and   $|I|=\lfloor\alpha n\rfloor$, we have
\begin{align*}
0
\leq
r_{1,n},r_{2,n}
\leq
1. 
\end{align*}
Hence, there exists a subsequence of
$\{(r_{0,n},r_{1,n},r_{2,n})\}_{n\geq1}$
that converges to a triple $(r_0,r_1,r_2)$. We also recall that   $(r_{0,n},r_{1,n},r_{2,n})$
is the rate tuple of an $(n-\lfloor\alpha n\rfloor)$-dimensional
unique-sum system, and 
$n-\lfloor\alpha n\rfloor\to\infty$ as $n\to\infty$.
Hence, by the definition of admissible triples, $(r_0,r_1,r_2)$ is
admissible for unique-sum systems.
It further follows from
\eqref{eq:pf-r0-epsilon} and
\eqref{eq:pf-total-epsilon} that
\begin{align*}
r_0
\leq
\frac{\alpha}{1-\alpha}
\quad\textup{and}\quad
r_0+r_1+r_2
\geq
\frac{1}{1-\alpha}
\bigg(
R+\frac12h(2\alpha)-h(\alpha)
\bigg).
\end{align*}

To complete the proof, it remains to prove the two inequalities \eqref{eq:pf-r0-epsilon} and \eqref{eq:pf-total-epsilon}.
By \eqref{bounded-eq1}, we immediately obtain
\begin{align*}
\limsup_{n\to\infty}r_{0,n}
\leq \limsup_{n\to\infty}\frac{\lfloor\alpha n\rfloor}{n-\lfloor\alpha n\rfloor}=
\frac{\alpha}{1-\alpha}.
\end{align*}
Furthermore, by the definitions of
$r_{0,n},r_{1,n},r_{2,n}$ (cf.~\eqref{eq:def-r0n}, \eqref{eq:def-r1n} and \eqref{eq:def-r2n}),
we consider
\begin{align}
 r_{0,n}+r_{1,n}+r_{2,n}
 &=\frac{1}{n-\lfloor\alpha n\rfloor}\log\big|\mathcal{V}_{I}^{(i,j)}\big|+\frac{1}{n-\lfloor\alpha n\rfloor}\log  2^{i} +\frac{1}{n-\lfloor\alpha n\rfloor}\log   2^{j}
 \nonumber\\
 &=
\frac{1}{n-\lfloor\alpha n\rfloor}
\log
\big(
\big|\mathcal{V}_{I}^{(i,j)}\big|\cdot
2^{i+j}
\big)\nonumber\\
&\geq  \frac{1}{n-\lfloor\alpha n\rfloor}
\log\Bigg[|\mbX||\mbY|\cdot
\frac{~\binom{\lfloor n/2\rfloor}{\lfloor\alpha n\rfloor}~}
{\binom{n}{\lfloor\alpha n\rfloor} }\cdot
\frac{1}{4(n-\lfloor\alpha n\rfloor+1)^2}\Bigg]
\label{eq:total-rate-finite-lower-new-2}
\\
&=
\frac{1}{n-\lfloor\alpha n\rfloor}
\log(|\mbX||\mbY|)
+
\frac{1}{n-\lfloor\alpha n\rfloor}
\log \binom{\lfloor n/2\rfloor}{\lfloor\alpha n\rfloor}
\nonumber\\
&~~~\;+
\frac{-1}{n-\lfloor\alpha n\rfloor}
\log   \binom{n}{\lfloor\alpha n\rfloor}
+
\frac{-1}{n-\lfloor\alpha n\rfloor}
\log\big(4(n-\lfloor\alpha n\rfloor+1)^2\big),
\label{eq:total-rate-finite-lower-new}
\end{align}
where
the inequality \eqref{eq:total-rate-finite-lower-new-2} follows from \eqref{eq:pf-extracted-triple-main}.
It follows from~\eqref{eq:total-rate-finite-lower-new} that
\begin{align}
\liminf_{n\to\infty}
\big(r_{0,n}+r_{1,n}+r_{2,n}\big)\geq
\liminf_{n\to\infty} \frac{1}{n-\lfloor\alpha n\rfloor}
\log(|\mbX||\mbY|)+\liminf_{n\to\infty}
\frac{1}{n-\lfloor\alpha n\rfloor}
\log \binom{\lfloor n/2\rfloor}{\lfloor\alpha n\rfloor}&
\nonumber\\
+\liminf_{n\to\infty}
\frac{-1}{n-\lfloor\alpha n\rfloor}
\log   \binom{n}{\lfloor\alpha n\rfloor}+\liminf_{n\to\infty}
\frac{-1}{n-\lfloor\alpha n\rfloor}
\log\big(4(n-\lfloor\alpha n\rfloor+1)^2\big).&\label{geq-eq}
\end{align}

Since $|\mbX||\mbY|=M(n)$ (cf.~\eqref{eq-M}) and
$
\lim_{n\to\infty}(1/n)\log M(n)=R
$ (cf.~\eqref{eq:def-Csum}),
we have
\begin{align}
\liminf_{n\to\infty} \frac{1}{n-\lfloor\alpha n\rfloor}
\log(|\mbX||\mbY|)=\lim_{n\to\infty}
\frac{1}{n-\lfloor\alpha n\rfloor}
\log(|\mbX||\mbY|)
=
\frac{R}{1-\alpha}.
\label{eq:pf-limit-Mn}
\end{align}
Furthermore, by the standard asymptotic estimates for binomial
coefficients, we have
\begin{align*}
\log
\binom{\lfloor n/2\rfloor}{\lfloor\alpha n\rfloor}
&=
\frac n2 h(2\alpha)+o(n); \\
\log
\binom{n}{\lfloor\alpha n\rfloor}
&=
nh(\alpha)+o(n).
\end{align*}
It then follows that
\begin{align}
\liminf_{n\to\infty}
\frac{1}{n-\lfloor\alpha n\rfloor}
\log \binom{\lfloor n/2\rfloor}{\lfloor\alpha n\rfloor}&=\lim_{n\to\infty}
\frac{1}{n-\lfloor\alpha n\rfloor}
\log
\binom{\lfloor n/2\rfloor}{\lfloor\alpha n\rfloor}
=
\frac{1}{2(1-\alpha)}h(2\alpha);
\label{eq:pf-limit-binomial-one}\\
 \liminf_{n\to\infty}
\frac{-1}{n-\lfloor\alpha n\rfloor}
\log   \binom{n}{\lfloor\alpha n\rfloor}&=\lim_{n\to\infty}
\frac{-1}{n-\lfloor\alpha n\rfloor}
\log
\binom{n}{\lfloor\alpha n\rfloor}
 =
\frac{-1}{1-\alpha}h(\alpha).
\label{eq:pf-limit-binomial-two}
\end{align}
Finally, we note that
\begin{align}
\liminf_{n\to\infty}
\frac{-1}{n-\lfloor\alpha n\rfloor}
\log\big(4(n-\lfloor\alpha n\rfloor+1)^2\big)=\lim_{n\to\infty}
\frac{-1}{n-\lfloor\alpha n\rfloor}
\log
\big(4(n-\lfloor\alpha n\rfloor+1)^2\big)
=
0.
\label{eq:pf-limit-polynomial}
\end{align}
Combining
\eqref{eq:pf-limit-Mn}, \eqref{eq:pf-limit-binomial-one}, \eqref{eq:pf-limit-binomial-two}, and \eqref{eq:pf-limit-polynomial} with \eqref{geq-eq}, we obtain
\begin{align*}
\liminf_{n\to\infty}
\big(r_{0,n}+r_{1,n}+r_{2,n}\big)
\geq
\frac{1}{1-\alpha}
\left(
R+\frac12h(2\alpha)-h(\alpha)
\right).
\end{align*}
This completes the proof of the proposition.
\end{proof}

\subsection{Proof of Lemma~\ref{lemma:coordinate-dyadic-selection}}\label{subsec-pf-lemma}

In this subsection, we prove
Lemma~\ref{lemma:coordinate-dyadic-selection}.
We first give some notation and auxiliary lemmas.
Recall from the discussion following
Lemma~\ref{lem:mean-xor-at-least-half} that  $n$ is an arbitrary positive integer, $(\mathcal X,\mathcal Y)$ is an $n$-dimensional unique-sum pair, and
$\boldsymbol{X}$ and $\boldsymbol{Y}$ are two independent random vectors
uniformly distributed over $\mathcal{X}$ and $\mathcal{Y}$, respectively, such that
\begin{align}
\E\big[\wt(\bX\oplus\bY)\big]\geq\frac n2.
\label{eq:section-balanced-assumption1}
\end{align}
For every subset $I\subseteq[n]$, define
\begin{align}
\Delta_{I}=
\big\{
(\boldsymbol{x},\boldsymbol{y})\in\mathcal{X}\times\mathcal{Y}:~\boldsymbol{y}[I]=\boldsymbol{1}-\boldsymbol{x}[I]
\big\} .
\label{eq:def-NnK-equiv-form}
\end{align}

\begin{lemma}\label{lem:coordinate-extraction}
For any integer $0\leq k\leq\lfloor n/2\rfloor$, there exists a subset
$I\subseteq[n]$ with $|I|=k$ such that
\begin{align*}
|\Delta_{I}|
\geq
|\mbX||\mbY|\cdot
\frac{~\binom{\lfloor n/2\rfloor}{k}~}
{\binom nk}.
\end{align*}
\end{lemma}

\begin{proof}
For notational simplicity, we define the following set:
\begin{align}
\mathcal{Q}_k
\triangleq
\Big\{
(I,\bx,\by)\in 2^{[n]}\times\mathcal{X}\times\mathcal{Y}:~|I|=k,\,
(\bx,\by)\in\Delta_I
\Big\},\label{def-S-k}
\end{align}
where $2^{[n]}$ denotes the power set of $[n]$, i.e., $2^{[n]}\triangleq\big\{I:\,I\subseteq[n]\big\}$.
Clearly, we have
\begin{align}
|\mathcal{Q}_k|=
\sum_{I\subseteq[n]\,\text{with}\,|I|=k}
|\Delta_I|.
\label{eq:Sk-count-I}
\end{align}
On the other hand, for every   pair
$(\boldsymbol{x},\boldsymbol{y})
\in\mathcal{X}\times\mathcal{Y}$,
we can see that the number of subsets
$I\subseteq [n]$ satisfying $|I|=k$
and  $\boldsymbol{y}[I]=\boldsymbol{1}-\boldsymbol{x}[I]$
is
\begin{align*}
\binom{
\mathrm{wt}(\boldsymbol{x}\oplus\boldsymbol{y})
}{
k
},
\end{align*}
where we note that  $\boldsymbol{y}[I]=\boldsymbol{1}-\boldsymbol{x}[I]$ if and only if
$\boldsymbol{x}[i]\neq \boldsymbol{y}[i]$ for all $i\in I$, and  use the convention that
$
\binom{a}{b}=0$ whenever $
0\leq a<b.
$
Together with the definition of $\mathcal{Q}_k$ (cf.~\eqref{def-S-k}),   we also have
\begin{align}
|\mathcal Q_k|
&=
\sum_{(\bx,\by)\in\mbX\times\mbY}
\binom{\wt(\bx\oplus\by)}{k}
\nonumber\\
&=|\mbX||\mbY|\cdot
\sum_{(\bx,\by)\in\mbX\times\mbY} \frac{1}{|\mbX||\mbY|}\cdot
\binom{\wt(\bx\oplus\by)}{k}
\nonumber\\
&=
|\mbX||\mbY|\cdot
\E\left[\binom{\wt(\boldsymbol{X}\oplus\boldsymbol{Y})}{k}\right],
\label{eq:Sk-count-pairs}
\end{align}
where we recall that $\boldsymbol{X}$ and $\boldsymbol{Y}$ are two independent random vectors
uniformly distributed over~$\mathcal{X}$ and $\mathcal{Y}$, respectively.

We note that
the number of subsets $I\subseteq[n]$ satisfying $|I|=k$ is
$\binom{n}{k}$. Then, by \eqref{eq:Sk-count-I}, there exists at least one subset
$I\subseteq[n]$ with $|I|=k$ such that
\begin{align*}
|\Delta_I|
\geq
\frac{~|\mathcal{Q}_k|~}{\binom{n}{k}}=|\mbX||\mbY|\cdot \frac{1}{~\binom{n}{k}~}\cdot
\E\left[\binom{\wt(\boldsymbol{X}\oplus\boldsymbol{Y})}{k}\right],
\end{align*}
where the  equality follows from \eqref{eq:Sk-count-pairs}.
To complete the proof, it suffices to prove that
\begin{align}
\E\left[\binom{\wt(\boldsymbol{X}\oplus\boldsymbol{Y})}{k}\right]
\geq
\binom{\lfloor n/2\rfloor}{k}.\label{pf-n-2}
\end{align}

We first note that the inequality \eqref{pf-n-2} holds for   $k=0$.
Now, we consider an arbitrary positive integer $1\leq k\leq \lfloor n/2\rfloor$.
For notational simplicity, we let $W\triangleq\wt(\boldsymbol{X}\oplus\boldsymbol{Y})$.
Clearly, the random variable $W$ takes values in the set
$\{0,1,\cdots,n\}$. Since $\E\big[\wt(\boldsymbol X\oplus\boldsymbol Y)\big]\geq  n/2$ (cf.~\eqref{eq:section-balanced-assumption1}), we   have
$\E[W]\geq n/2$.
Define
\begin{align*}
f_k(w)
\triangleq
\binom wk,\qquad\forall~w=0,1,2,\cdots.
\end{align*}
By Pascal's identity, we have
\begin{align}
f_k(w+1)-f_k(w)
&=
\binom{w+1}{k}-\binom{w}{k}=
\binom{w}{k-1},\qquad\forall~w=0,1,2,\cdots.
\label{eq:pf-coordinate-first-difference}
\end{align}
We claim that
\begin{align}
f_k(w)
\geq
f_k\big(\lfloor n/2\rfloor\big)+
\big(w-\lfloor n/2\rfloor\big)\binom{\lfloor n/2\rfloor}{k-1},\qquad\forall~0\leq w\leq n.
\label{eq:pf-coordinate-linear-bound}
\end{align}
To prove this claim, we consider the following three cases.

\medskip
\noindent
\textbf{Case 1:} $0\leq  w\leq \lfloor n/2\rfloor-1$.
\medskip

For this case, we have
\begin{align}
f_k(w)-f_k\big(\lfloor n/2\rfloor\big)
&=
-\sum_{\ell=w}^{\lfloor n/2\rfloor-1}
\big( f_k(\ell+1)-f_k(\ell)\big)
\nonumber\\
&=
-\sum_{\ell=w}^{\lfloor n/2\rfloor-1}
\binom{\ell}{k-1}
\label{eq:pf-coordinate-small-w-1}\\
&\geq-\sum_{\ell=w}^{\lfloor n/2\rfloor-1} \binom{\lfloor n/2\rfloor}{k-1}\label{eq:pf-coordinate-small-w-2}\\
&=
\big(w-\lfloor n/2\rfloor\big)\binom{\lfloor n/2\rfloor}{k-1},\nonumber
\end{align}
where the equality \eqref{eq:pf-coordinate-small-w-1} follows from
\eqref{eq:pf-coordinate-first-difference}, and the inequality \eqref{eq:pf-coordinate-small-w-2} follows from the fact that
$\binom{\ell}{k-1}\leq
\binom{\lfloor n/2\rfloor}{k-1}$ for any $w\leq  \ell\leq \lfloor n/2\rfloor-1$.

\medskip
\noindent
\textbf{Case 2:} $w=\lfloor n/2\rfloor$.
\medskip

For this case, it is easy to see that the inequality \eqref{eq:pf-coordinate-linear-bound} holds.

\medskip
\noindent
\textbf{Case 3:} $\lfloor n/2\rfloor+1\leq w\leq n$.
\medskip

For this case, we have
\begin{align}
f_k(w)-f_k\big(\lfloor n/2\rfloor\big)
&=
\sum_{\ell=\lfloor n/2\rfloor}^{w-1}
\big( f_k(\ell+1)-f_k(\ell)\big)
\nonumber\\
&=
\sum_{\ell=\lfloor n/2\rfloor}^{w-1}
\binom{\ell}{k-1}
\label{eq:pf-coordinate-large-w-1}\\
&\geq
\sum_{\ell=\lfloor n/2\rfloor}^{w-1}
\binom{\lfloor n/2\rfloor}{k-1}
\label{eq:pf-coordinate-large-w-2}\\
&=
\big(w- \lfloor n/2\rfloor\big)\binom{\lfloor n/2\rfloor}{k-1},\nonumber
\end{align}
where the equality \eqref{eq:pf-coordinate-large-w-1} follows from \eqref{eq:pf-coordinate-first-difference}, and
the inequality \eqref{eq:pf-coordinate-large-w-2} follows because $\binom{\ell}{k-1}\geq
\binom{\lfloor n/2\rfloor}{k-1}$ for any $\lfloor n/2\rfloor\leq  \ell\leq w-1$.

By
\eqref{eq:pf-coordinate-linear-bound}, we further have
\begin{align*}
\E\big[f_k(W)\big]
&\geq\E\bigg[
f_k\big(\lfloor n/2\rfloor\big)+
\big(W-\lfloor n/2\rfloor\big)\binom{\lfloor n/2\rfloor}{k-1}\bigg]\\
&=
f_k\big(\lfloor n/2\rfloor\big) +
\big(\E[W]-\lfloor n/2\rfloor\big)\binom{\lfloor n/2\rfloor}{k-1}
\\
&\geq
f_k\big(\lfloor n/2\rfloor\big),
\end{align*}
where the last inequality follows from
$\E[W]\geq n/2\geq\lfloor n/2\rfloor$.
The lemma is thus proved.
\end{proof}

\begin{lemma}\label{lem:dyadic-section}
For any subset $I\subseteq[n]$,
there exist integers $0\leq i,j\leq n-|I|$ such that
\begin{align}
\big|\mathcal{V}_{I}^{(i,j)}\big|\cdot 2^{i+j}
\geq\frac{|\Delta_{I}|}{4(n-|I|+1)^2}.
\label{eq:dyadic-section-lower}
\end{align}
\end{lemma}

\begin{proof}
If $|\Delta_{I}|=0$, then
\eqref{eq:dyadic-section-lower} holds
for any $0\leq i,j\leq n-|I|$.
Hence, in the following, we only need to consider the case
$|\Delta_{I}|>0$, i.e., there exists a $\boldsymbol{v}\in\{0,1\}^{|I|}$ such that
$\mathcal{X}_{I,\boldsymbol{v}}$ and
$\mathcal{Y}_{I,\boldsymbol{1}-\boldsymbol{v}}$
are both nonempty.

Recalling the definition of $\Delta_I$ in
\eqref{eq:def-NnK-equiv-form}, we have
\begin{align*}
&\Delta_{I}=\big\{
(\boldsymbol{x},\boldsymbol{y})\in\mathcal{X}\times\mathcal{Y}:~\boldsymbol{y}[I]=\boldsymbol{1}-\boldsymbol{x}[I]
\big\}\nonumber\\
&=
\bigsqcup_{\boldsymbol{v}\in\{0,1\}^{|I|}}
\big\{
(\boldsymbol{x},\boldsymbol{y})\in
\mathcal{X} \times\mathcal{Y} :~
\boldsymbol{x}[I]=\boldsymbol{v},~
\boldsymbol{y}[I]=\boldsymbol{1}-\boldsymbol{v}
\big\}\nonumber\\
&=
\bigsqcup_{\boldsymbol{v}\in\{0,1\}^{|I|}}
\big\{
 \boldsymbol{x}\in
\mathcal{X}:~
\boldsymbol{x}[I]=\boldsymbol{v}
\big\}\times \big\{
 \boldsymbol{y} \in \mathcal{Y} :~
\boldsymbol{y}[I]=\boldsymbol{1}-\boldsymbol{v}
\big\},
\end{align*}
where  ``$\,\sqcup\,$'' denotes the disjoint union.
This further implies that
\begin{align}
&|\Delta_{I}|=\sum_{\boldsymbol{v}\in\{0,1\}^{|I|}}
\big|\big\{
 \boldsymbol{x}\in
\mathcal{X}:~
\boldsymbol{x}[I]=\boldsymbol{v}
\big\}\big|\cdot \big|\big\{
 \boldsymbol{y} \in \mathcal{Y} :~
\boldsymbol{y}[I]=\boldsymbol{1}-\boldsymbol{v}
\big\}\big|\nonumber\\
&=\sum_{\boldsymbol{v}\in\{0,1\}^{|I|}}
\big|\big\{\boldsymbol{x}[I^{\mathrm{c}}]:~
\boldsymbol{x}\in\mathcal{X}~\textup{such that}~
\boldsymbol{x}[I]=\boldsymbol{v}\big\}\big|\cdot \big|\big\{\boldsymbol{y}[I^{\mathrm{c}}]:~
\boldsymbol{y}\in\mathcal{Y}~\textup{such that}~
\boldsymbol{y}[I]=\boldsymbol{1}-\boldsymbol{v}\big\}\big|\nonumber\\
&=
\sum_{\boldsymbol{v}\in\{0,1\}^{|I|}}
\big|\mathcal{X}_{I,\boldsymbol{v}}\big|
\big|\mathcal{Y}_{I,\boldsymbol{1}-\boldsymbol{v}}\big|,
\label{pf-dyadic-section-eq1}
\end{align}
where the equality \eqref{pf-dyadic-section-eq1} follows from   the definitions of
$\mathcal X_{I,\boldsymbol v}$ and
$\mathcal Y_{I,\boldsymbol1-\boldsymbol v}$ (cf.~\eqref{eq:def-An-t} and \eqref{eq:def-Bn-t}).
For every
$\boldsymbol{v}\in\{0,1\}^{|I|}$ such that
$\mathcal{X}_{I,\boldsymbol{v}}$ and
$\mathcal{Y}_{I,\boldsymbol{1}-\boldsymbol{v}}$
are both nonempty, since $\mathcal{X}_{I,\boldsymbol{v}}$ and
$\mathcal{Y}_{I,\boldsymbol{1}-\boldsymbol{v}}$
are both subsets of $\{0,1\}^{|I^{\mathrm{c}}|}$ ($=\{0,1\}^{n-|I|}$), we have
\begin{align*}
1
&\leq
\big|\mathcal{X}_{I,\boldsymbol{v}}\big|
\leq
2^{n-|I|}\quad\text{ and }\quad
1 \leq
\big|\mathcal{Y}_{I,\boldsymbol{1}-\boldsymbol{v}}\big|
\leq
2^{n-|I|}.
\end{align*}
Therefore, the nonempty sets
$
\mathcal{V}_{I}^{(i,j)},$ $
0\leq i,j\leq n-|I|$  
form a partition of all
$\boldsymbol{v}\in\{0,1\}^{|I|}$ for which
$\mathcal{X}_{I,\boldsymbol{v}}$ and
$\mathcal{Y}_{I,\boldsymbol{1}-\boldsymbol{v}}$
are both nonempty (see~\eqref{eq:def-Tnij} for the definition of $\mathcal{V}_{I}^{(i,j)}$).
This, together with~\eqref{pf-dyadic-section-eq1}, further implies  that
\begin{align}
|\Delta_{I}|
&=
\sum_{i=0}^{n-|I|}\;
\sum_{j=0}^{n-|I|}\;
\sum_{\boldsymbol{v}\in\mathcal{V}_{I}^{(i,j)}}
\big|\mathcal{X}_{I,\boldsymbol{v}}\big|
\big|\mathcal{Y}_{I,\boldsymbol{1}-\boldsymbol{v}}\big|.
\label{pf-dyadic-section-eq2}
\end{align}
We note that
 there are
$(n-|I|+1)^2$ pairs $(i,j)$ with $
0\leq i,j\leq n-|I|$. It then follows from
\eqref{pf-dyadic-section-eq2} that there exists a pair $(i,j)$ with
$
0\leq i,j\leq n-|I|$
such that
\begin{align}
&\sum_{\boldsymbol{v}\in\mathcal{V}_{I}^{(i,j)}}
\big|\mathcal{X}_{I,\boldsymbol{v}}\big|
\big|\mathcal{Y}_{I,\boldsymbol{1}-\boldsymbol{v}}\big|\geq
\frac{|\Delta_{I}|}{(n-|I|+1)^2}.
\label{pf-dyadic-section-eq3}
\end{align}
Furthermore,
for every
$\boldsymbol{v}\in\mathcal{V}_{I}^{(i,j)}$,
by the definition of
$\mathcal{V}_{I}^{(i,j)}$, we have
\begin{align}
&
4\cdot2^{i+j}=
2^{i+1}\cdot2^{j+1}>
\big|\mathcal{X}_{I,\boldsymbol{v}}\big|
\big|\mathcal{Y}_{I,\boldsymbol{1}-\boldsymbol{v}}\big|.\label{pf-dyadic-section-eq4}
\end{align}
By \eqref{pf-dyadic-section-eq3} and
\eqref{pf-dyadic-section-eq4}, we obtain
\begin{align*}
\big|\mathcal{V}_{I}^{(i,j)}\big|\cdot
4\cdot2^{i+j} &=\sum_{\boldsymbol{v}\in\mathcal{V}_{I}^{(i,j)}}
4\cdot2^{i+j}\\
&>\sum_{\boldsymbol{v}\in\mathcal{V}_{I}^{(i,j)}}
\big|\mathcal{X}_{I,\boldsymbol{v}}\big|
\big|\mathcal{Y}_{I,\boldsymbol{1}-\boldsymbol{v}}\big|\\
&\geq
\frac{|\Delta_{I}|}{(n-|I|+1)^2}.
\end{align*}
This immediately proves \eqref{eq:dyadic-section-lower}, and the lemma is thus proved.
\end{proof}

We now prove Lemma~\ref{lemma:coordinate-dyadic-selection}.

\begin{proof}[Proof of Lemma~\ref{lemma:coordinate-dyadic-selection}]
Fix an arbitrary integer $0\leq k\leq\lfloor n/2\rfloor$. It first follows from Lemma~\ref{lem:coordinate-extraction}
that  there exists a subset
$I\subseteq[n]$ with $|I|=k$ such that
\begin{align}
|\Delta_{I}|
\geq
|\mbX||\mbY|\cdot
\frac{~\binom{\lfloor n/2\rfloor}{k}~}
{\binom nk}.
\label{pf:lemma-eq1}
\end{align}
Further, by Lemma~\ref{lem:dyadic-section},
there exist integers $0\leq i,j\leq n-|I|$ such that
\begin{align}
\big|\mathcal{V}_{I}^{(i,j)}\big|\cdot 2^{i+j}
\geq\frac{|\Delta_{I}|}{4(n-|I|+1)^2}=\frac{|\Delta_{I}|}{4(n-k+1)^2},
\label{pf:lemma-eq2}
\end{align}
where the equality follows from $|I|=k$.
Combining \eqref{pf:lemma-eq1} and \eqref{pf:lemma-eq2}, we obtain
\begin{align*}
\big|\mathcal{V}_{I}^{(i,j)}\big|\cdot 2^{i+j}
\geq\frac{|\Delta_{I}|}{4(n-k+1)^2}\geq
|\mbX||\mbY|\cdot
\frac{\binom{\lfloor n/2\rfloor}{k}}
{\binom nk}\cdot
\frac{1}{4(n-k+1)^2}.
\end{align*}
The lemma is thus proved.
\end{proof}

\section{Proof of Theorem~\ref{thm:main}}
\label{sec:analytic-bound}

In this section,
we  prove Theorem~\ref{thm:main}.  Fix an arbitrary real number $0<\alpha\leq1/4$.
By Proposition~\ref{prop:extracted-triple}, there exists an admissible triple
$(r_0,r_1,r_2)$ for unique-sum systems satisfying
\begin{align}
r_0
\leq
\frac{\alpha}{1-\alpha}\quad\textup{ and }\quad
r_0+r_1+r_2
\geq
\frac{1}{1-\alpha}\bigg(R+\frac12h(2\alpha)-h(\alpha)\bigg).\label{eq:main-proof-total-lower}
\end{align}
Clearly, $
0< \alpha/(1-\alpha)\leq 1/3
$ by  $0<\alpha\leq1/4$.
Applying Proposition~\ref{cor:Gamma-system-bound} with $\beta
=\alpha/(1-\alpha)$,
we have
\begin{align}
r_0+r_1+r_2
\leq
L\left(
\Gamma^{-1}
\Big(
\frac{\alpha}{1-\alpha}
\Big)
\right).
\label{eq:main-proof-total-upper}
\end{align}
By~\eqref{eq:main-proof-total-lower} and
\eqref{eq:main-proof-total-upper}, we obtain
\begin{align}
R
&\leq
h(\alpha)
-\frac12h(2\alpha)
+
(1-\alpha)
L\bigg(
\Gamma^{-1}
\Big(
\frac{\alpha}{1-\alpha}
\Big)
\bigg)\nonumber\\
&=h(\alpha)
-\frac{1}{2}h(2\alpha)
+
(1-\alpha)
\bigg[
h\bigg(
\Gamma^{-1}
\Big(
\frac{\alpha}{1-\alpha}
\Big)
\bigg)
+
1
-
\Gamma^{-1}
\Big(
\frac{\alpha}{1-\alpha}
\Big)
\bigg]=R_{\alpha},
\label{eq:main-alpha-bound}
\end{align}
where the first equality in \eqref{eq:main-alpha-bound}
follows from the definition of  $
 L
$ in~\eqref{eq:def-L}, and the last equality in~\eqref{eq:main-alpha-bound} follows from the definition of $R_{\alpha}$ in~\eqref{def-R-alpha}.
Since \eqref{eq:main-alpha-bound} holds for every
$0<\alpha\leq1/4$, we have
\begin{align*}
R
\leq
\inf_{0<\alpha\leq1/4}R_\alpha.
\end{align*}
The theorem is thus proved.

\end{document}